\documentclass[12pt]{amsart}

\usepackage{amsmath,amssymb,amsthm,enumitem}

\usepackage{multirow}
\usepackage{tikz}
\usepackage{amsfonts}

\usetikzlibrary{matrix,arrows,decorations.pathmorphing,decorations.pathreplacing}

\usepackage{enumitem}
\usepackage{soul}

\makeatletter
\@namedef{subjclassname@2010}{%
  \textup{2010} Mathematics Subject Classification}
\makeatother

\DeclareMathOperator{\Gal}{Gal}
\DeclareMathOperator{\Res}{Res}

\def\Tr{\operatorname{Tr}}

\def \F {\mathbb F}
\renewcommand{\phi}{\varphi}

\usepackage{todonotes}

\newtheorem{theorem}{Theorem}[section]
\newtheorem*{thm}{Theorem}

\newtheorem{proposition}[theorem]{Proposition}
\newtheorem{lemma}[theorem]{Lemma}

\theoremstyle{definition}
\newtheorem{remark}[theorem]{Remark}

\def\cqfd{
{\hfill
\kern 6pt\penalty 500
\raise -1pt\hbox{\vrule\vbox to 5pt{\hrule width 4pt
\vfill\hrule}\vrule}}
\break}
\font\tengoth=eufm10
\font\sevengoth=eufm7
\font\fivegoth=eufm5
\newfam\gothfam
\textfont\gothfam=\tengoth\scriptfont\gothfam=\sevengoth\scriptscriptfont\gothfam=\fivegoth

\title[Maximal differential uniformity]{Polynomials with maximal differential uniformity and the exceptional APN conjecture}

\author[Aubry]{Yves Aubry}
\address[Aubry]{Institut de Math\'ematiques de Toulon - IMATH, Universit\'e de Toulon, France}
\address[Aubry]{Institut de Math\'ematiques de Marseille - I2M, Aix Marseille Univ, CNRS, Centrale Marseille, France}
\email{yves.aubry@univ-tln.fr}

\author[Herbaut]{Fabien Herbaut}
\address[Herbaut]{INSPE Nice-Toulon, Universit\'e C\^ote d'Azur, France}
\address[Herbaut]{Institut de Math\'ematiques de Toulon - IMATH, Universit\'e de Toulon, France}

\email{fabien.herbaut@univ-cotedazur.fr}

\author[Issa]{Ali Issa}
\address[Issa]{Institut de Math\'ematiques de Toulon - IMATH, Universit\'e de Toulon, France}
\email{ali.issa@univ-tln.fr}

\begin{document} 

\footnote{This work is partially supported by the French Agence Nationale de la Recherche through the SWAP project under Contract ANR-21-CE39-0012.}

%%%%% To ease editing, for IMPAN journals add:

\baselineskip=17pt

%%%%%%%%%%%%%%%%

\begin{abstract}
We contribute to the 
exceptional APN conjecture
by
showing that
no polynomial of degree
$m=2^{r}(2^{\ell}+1)$ where $\gcd(r,\ell)\leqslant 2$,
  $r \geqslant 2$, $\ell\geqslant 1$ 
  with a nonzero second leading coefficient 
  can be 
  APN over infinitely many extensions of the base field.
More precisely, we prove that
for $n$ sufficiently large,
all polynomials of ${\mathbb F}_{2^n}[x]$
of such a degree  with a nonzero second leading coefficient have a
differential uniformity  equal to $m-2$.

\end{abstract}

\date{\today}



\maketitle

\section{Introduction}

\subsection{Statement of results}
The notion of differential uniformity is introduced by
Nyberg in \cite{Nyberg}
to measure
the resistance 
of mappings between finite Abelian groups
against differential cryptanalysis.
In the context of a finite field ${\mathbb F}_q$  
one defines
the differential uniformity of a polynomial $f \in {\mathbb F}_{q}[x]$  
as
the greatest number of solutions 
of the set of equations $f(x+ \alpha) - f(x) = \beta$
where $\alpha$
and $\beta$ belong to ${\mathbb F}_q$ with $\alpha$ nonzero:
$$\delta_{{\mathbb F}_q}(f):=\max_{(\alpha,\beta)\in{\mathbb F}_q^{\ast}\times{\mathbb F}_q}\sharp\{x\in{\mathbb F}_{q} \mid f(x+\alpha)-f(x)=\beta\}.$$

Particular emphasis
is being placed on the even characteristic 
and this is the framework 
that we will consider here.
The differential uniformity is then obviously even and  its smallest value is  2.  
Polynomials $f \in {\mathbb F}_{2^n}[x]$ such that $\delta_{{\mathbb F}_{2^n}}(f)=2$
are highly relevant in cryptography and
are called APN polynomials (for Almost Perfect Nonlinear).

APN polynomials which are 
 APN over infinitely many extensions of the base field
have also attracted some attention
and they are called exceptional APN.
Dillon has conjectured in \cite{Dillon}
that the only monomials
among them
have degrees 
 $2^k+1$ and $2^{2k}-2^k+1$ 
(which  
are called Gold and Kasami-Welch numbers respectively).
The conjecture has been 
 resolved  
by Hernando and McGuire in \cite{H-McG}.

Thereafter,
the first author, McGuire and Rodier 
have conjectured  in \cite{A-McG-R} that
up to the CCZ equivalence
(an equivalence relation
introduced by Carlet, Charpin and Zinoviev in \cite{CCZ} and discussed in 
\cite{BCP})
these 
 monomials 
 are the only exceptional APN {\sl polynomials}.
 This statement is now referred to as the Aubry-McGuire-Rodier conjecture.
In this direction they
have established
that if the degree of a polynomial $f$ 
is odd
but
neither a Gold nor 
a Kasami-Welch number,
 then $f$ is not exceptional APN.
From there some authors have focused
on the study 
of polynomials of degree Gold  or Kasami-Welch
(see for instance the survey  \cite{Delgado} of Delgado).

Few is known 
about the even degree case.
The first author, McGuire and Rodier
have proved   in \cite{A-McG-R} that
if $f$ is a polynomial of degree $2e$
with $e$ odd
and  if
$f$ contains a term of odd degree
then 
$f$ is not exceptional APN.
Moreover, Bartoli and Schmidt have stated 
in Proposition 1.4 in \cite{B-S} 
that if a polynomial of
even degree $m$ is exceptional APN then $m\equiv 0 \pmod 4$.

The case where $f$ has degree $4e$ with $e\equiv 3 \pmod 4$ 
and satisfies an extra condition
 has been 
 studied by Rodier in \cite{Rodier}.
Caullery has handled in \cite{Caullery} the case where 
$f$ has degree $4e$ with $e>3$ such that $\phi_e$ is absolutely irreducible, 
where
$$\phi_e(x,y,z):=\frac{x^e+y^e+z^e +(x+y+z)^e}{(x+y)(x+z)(y+z)}.$$
Results  on the absolute irreducibility of the polynomials $\phi_e$ are for example compiled in
\cite{J-McG-W}.
 As pointed in the proof of Lemma 2.2 in  \cite{A-McG-R}
the polynomial $\phi_e$ is not absolutely irreducible when $e$ is even,
so the case of polynomials of degree $m\equiv 0 \pmod 8$ is still open.
Moreover, the polynomial $\phi_e$ is not absolutely  irreducible if $e$ is a Gold number as shown by
Janwa and Wilson in Theorem 4 in \cite{J-W} 
so the case of degree $4e$ with $e$ a Gold number is also still open.

\medskip

Rather, one can ask how large the differential uniformity can be.
Unless
$f$ is an additive polynomial plus a constant,
the maximal value 
that the differential uniformity of a
polynomial $f$ of degree $m$
can reach is the degree of the derivative $f(x+\alpha)-f(x)$
which is bounded by
$m-1$
if $m$ is odd and by $m-2$ otherwise.
Consequently we will say that a polynomial has a maximal differential uniformity
if this bound is reached.

A density result has first been established  in this direction
in  \cite{Felipe} where
Voloch
has proved 
that {\sl most} polynomials of degree $m$  congruent to $0$ or $3$ modulo $4$
achieve this maximal differential uniformity.
One can also find a generalization to the second order differential uniformity in \cite{YvesFabien}.

Moreover, 
Voloch and the two first authors
provide in Theorems 5.3 and 5.5 in \cite{AHV} explicit infinite families of
odd integers $m$
such that 
{\sl all} polynomials of degree $m$ (and not just {\sl most of them})
have a
maximal differential uniformity for $n$ large enough.

The main  purpose of this paper is to extend these results to
 infinitely many explicit even degrees.
 
\begin{thm}(Theorem \ref{theorem:main})
Let  $m=2^{r}(2^{\ell}+1)$ where $\gcd(r,\ell)\leqslant 2$,
  $r \geqslant 2$ and $\ell\geqslant 1$.
For $n$ sufficiently large,
for all polynomials $f=\sum_{k=0}^{m}a_{m-k} x^{k}\in{\mathbb F}_{2^n}[x]$
of degree $m$ such that $a_1\not=0$ 
 the differential uniformity $\delta(f)$ is maximal
that is $\delta(f)=m-2$.

In particular, no polynomial
$f=\sum_{k=0}^{m}a_{m-k} x^{k}$ 
of such a degree and
 such that $a_1 \neq 0$
can be exceptional APN.

\end{thm}

 This gives  contributions to the exceptional APN conjecture for the two open cases mentioned above.
 Indeed it almost
solves 
the case of the degrees $m=4e$ where $e$ is a Gold number and it is a first step  for the degrees $m\equiv 0\pmod 8$.

Note that
point {\it{(i)}}
 of
Theorem 2  in \cite{Felipe}  reveals that
 the condition $a_1\not=0$ in  the previous 
theorem
is necessary to get  the maximality of $\delta(f)$.

\bigskip

It is  worth stressing that the methods used until now 
to prove that polynomials are not exceptional APN
have rested on algebraic geometric tools.
The point was to apply 
Weil type bounds
for the number of rational points 
on varieties defined over finite fields.
In contrast, our approach comes from algebraic number theory.
The point here is to apply the Chebotarev density theorem for functions fields introduced in this context by Voloch in \cite{Felipe}.

%%%%%%%%%%%%%%%%%%%

%%%%%%%%%%%%%%%%%%%

%%%%%%%%%%%%%%%%%%%
%%%%%%%%%%%%%%%%%%%
%%%%%%%%%%%%%%%%%%
%%%%%%%%

\subsection{Context and method of proof}\label{subsection:strategy}

Before entering into details in the next section
we provide here
the very broad lines of our approach
 which
 involves
  the Chebotarev density theorem, monodromy groups  and Morse polynomials.

In the whole paper we will consider
a polynomial 
$f$ of  ${\mathbb F}_{q}[x]$
where $q=2^n$ 
and we will denote by $\overline{\mathbb F}_2$ an algebraic closure of ${\mathbb F}_2$.
For any $\alpha\in{\mathbb F}_q^{\ast}$ 
the derivative of $f$ with respect to $\alpha$
will be denoted
by  $D_{\alpha}f(x):=f(x+\alpha)+f(x)$.

Let us recall the following explicit version of the Chebotarev density theorem for first degree primes given by Pollack in
\cite{Pollack}.

\begin{theorem} (\textbf{Chebotarev})\label{Chebotarev}
Suppose that $\Omega/{\mathbb F}_q(t)$ is a Galois extension having full field of constants ${\mathbb F}_{q^D}$.
Let $\mathcal C$ be a conjugacy class of $\Gal(\Omega/{\mathbb F}_q(t))$ every element of which restricts down to the $q$th power map on ${\mathbb F}_{q^D}$.
Let 
$V({\mathcal C})$ be the number of 
first degree primes $v$  of  ${\mathbb F}_q(t)$ unramified in $\Omega$ such that the Artin symbol $\Bigl(\frac{\Omega/{\mathbb F}_q(t)}{v}\Bigr)$ equals ${\mathcal C}$.
Then 
$$\left \lvert V({\mathcal C})-\frac{\sharp{\mathcal C}}{[\Omega:{\mathbb F}_{q^D}(t)]}q\right\rvert 
\leqslant 2 \frac{\sharp{\mathcal C}}{[\Omega:{\mathbb F}_{q^D}(t)]} \bigl(gq^{1/2} + g + [\Omega:{\mathbb F}_{q^D}(t)]\bigr)$$
where $g$ denotes the genus of $\Omega/{\mathbb F}_{q^D}$.
\end{theorem}

We will see that it
 ensures
that if the geometric and arithmetic monodromy groups of $D_{\alpha} f$
coincide
then for $n$ sufficiently large
there exists
 $\beta \in {\mathbb F}_{2^n}$
such that the number of solutions  of the equation $D_{\alpha}f(x)=\beta$
is equal to the degree of $D_{\alpha}f$.

But how can we compare the monodromy groups of $D_{\alpha} f$?
Denote by $\Omega$  
the splitting field 
of the
polynomial $D_{\alpha}f(x)-t$ over the field ${\mathbb F}_q(t)$ with $t$ transcendental over  ${\mathbb F}_q$.
The method developed in \cite{Felipe} consists in introducing
an intermediate  field $F$
between 
$\Omega$ and ${\mathbb F}_q(t)$, namely
 the splitting field   of the
polynomial $L_{\alpha}f(x)-t$ over the field ${\mathbb F}_q(t)$,
where
$L_{\alpha}f$  is
the unique polynomial 
such that
$L_{\alpha}f \left( x(x+\alpha)\right)=D_{\alpha}f(x)$
 (see Proposition 2.3 of  \cite{AHV} for the existence and the unicity of such a polynomial $L_{\alpha}f$).

The cornerstone of the results obtained in \cite{Felipe}
is that
for almost all $f$ 
of a given degree the associated polynomial
$L_{\alpha}f $ is Morse.
Also, it is proven in \cite{AHV} that
for some specific degrees $m$,
for any polynomial $f$ of degree $m$ 
there exists $\alpha$ such that
$L_{\alpha}f $ is  Morse.
Recall that a polynomial $g\in{\mathbb F}_{2^n}[x]$ is said to be Morse
(see the Appendix of Geyer to 
 \cite{JardenRazon})
if the critical points of $g$ are nondegenerate 
(i.e. the derivative $g'$ and the second Hasse-Schmidt derivative $g^{[2]}$ have no common roots),
if the critical values of $g$ are distincts 
(different zeros of $g'$  give different values of $g$) and
if the degree of $g$ is prime to the characteristic.

The reason we are interested in Morse polynomials  
comes from 
the more general form of the Hilbert theorem given by Serre in Theorem 4.4.5 of \cite{Serre}
and outlined in even characteristic in the Appendix of Geyer in \cite{JardenRazon}) which asserts that  the 
geometric 
monodromy group of a Morse polynomial 
is the full symmetric group.

It remains to identify when
the polynomial
$g:=L_{\alpha}f $ is  Morse.
First, 
the resultant between $g'$ and $g^{[2]}$
is a classical tool 
to recognize polynomials $g$ with
nondegenerate
critical points.
Thus,
the main difficulty rests in
the study of polynomials with distinct critical values.
To this end we make use
of the algebraic characterization of such polynomials
obtained by Geyer in the same appendix.
Last, the parity condition on the degree of $g$
explains the common hypothesis of the results of this paper:
when $f$ has even degree,
the polynomial $L_{\alpha}f$ will have odd degree
as soon as 
the degree of $f$ is divisible by $4$
and
its second leading coefficient is nonzero.

In order to explain how to treat
the extension $\Omega / F$,
let us  write $L_{\alpha} f = \sum_{k=0}^d b_{d-k} x^{k}$.
Proposition 4.6 in \cite{AHV}
states that if $L_{\alpha} f$ is Morse and if
$x^2+\alpha x = b_1/b_0$ has a solution in $\mathbb{F}_q$
then the extension $\Omega/F$ is regular.
As we know simple expressions of $b_1$ and $b_0$
in terms of the coefficients of $f$ and $\alpha$, 
 Hilbert's Theorem 90 enables
to translate the problem into a polynomial equation.

The following diagram sums up 
the different conditions we intend to verify
to prove that the extensions are regular.

\begin{figure}[h]\label{diagram}
\begin{tikzpicture}
[node distance=1.5cm]

 \node (Fqdet)          {${\mathbb F}_{2^n}(t)$};
 \node (Fqdetd) [right of=  Fqdet] {};
\node (Fqdetdd) [right of =Fqdetd] {};
 \node (vide) [above of=Fqdet] {};
\node (F)   [above of=vide]   {$F$};

\node (Fd)   [right of=F]   {}; %%% Test 2021
\node (Fdd)   [right of=Fd]   {}; %%% Test 2021

 \node (FqFdet) [right of=vide] {${\mathbb F}_{2^n}^{F}(t)$};
 
  \node (vide2) [above of=F] {};
\node (Omega)   [above of=vide2]   {$\Omega$};
\node (Omegad)   [right of=Omega]   {}; %%% Test 2021
\node (Omegadd)   [right of=Omegad]   {}; %%% Test 2021
%\node at (0.7,-2) [right of=F,right of=Omega] {$\}$};
%(Omega) -- (F) node [black,midway,xshift=0.8cm] {\footnotesize$P_2$};
 \node (FqOmegadeF) [right of=vide2] {$F{\mathbb F}_{2^n}^{\Omega}$};
% \node(equality) [right of=FqOmegadeF] {$  = {\mathbb F}_q^{\Omega}(t)$};
\draw[decorate,decoration=brace] (Omegadd)--(Fdd)[right, midway,scale=0.9]
node[above=0.2cm,right=0.7cm,pos=0.1] 
{\scriptsize{
Regular as soon as: }}
node[above=0.2cm,right=0.7cm,pos=0.4] 
{\scriptsize{
the extension below is regular and when}}
node[above=0.2cm,right=0.7cm,pos=0.6] 
{\scriptsize{
\textbf{(II)} there exists $x \in \mathbb{F}_{{2^n}}$ s.t $x^2+\alpha x = \frac{b_1}{b_0}$}
};

\draw[decorate,decoration=brace] (Fdd)--(Fqdetdd)[right, midway,scale=0.9]
node[above=0.2cm,right=0.7cm,pos=0.1] 
{\scriptsize{
Regular as soon as $L_{\alpha}f$ is Morse i.e. when}}
node[above=0.2cm,right=0.7cm,pos=0.4]
{\scriptsize{
\textbf{(I.a)}  $L_{\alpha}f$ has nondegenerate critical points}}
node[above=0.2cm,right=0.7cm,pos=0.6] 
{\scriptsize{
\textbf{(I.b)}  $L_{\alpha}f$ has different critical values}}
node[above=0.2cm,right=0.7cm,pos=0.8] 
{
\scriptsize{
\textbf{(I.c)}   $L_{\alpha}f$ has odd degree
}
};

  \draw (Fqdet) to node[left, midway,scale=0.9]  {$G=\Gal\left(F/{\mathbb F}_{2^n}(t)\right)$} (F);
  \draw (Fqdet)--(FqFdet);
  \draw (FqFdet) to node[right, midway,scale=0.9]  {\ $\overline G$} (F);
  
  \draw (F) to node[left, midway,scale=0.9]  {$\Gamma=\Gal \left( \Omega / F \right) $} (Omega);
  \draw (F)--(FqOmegadeF);
  \draw (FqOmegadeF) to node[right, midway,scale=0.9]  {\ $\overline \Gamma$} (Omega);
  \end{tikzpicture}
\caption{}
  \label{diagram}
  \end{figure}
As explained above, condition 
(I.c) will involve 
a congruence condition on the degree $m$ of $f$,
and the non-vanishing of the second leading coefficient of $f$.
Conditions (I.a), (I.b) and (II) will translate into algebraic equations.
For the specific degrees mentioned in the introduction
we will manage
to bound the number of $\alpha$ for which 
one of the conditions fails.

%%%%%%%%%%%%%%%%%%%

%%%%%%%%%%%%%%%%%%%

%%%%%%%%%%%%%%%%%%%

%%%%%%%%%%%%%%%%%%%

%%%%%%%%%%%%%%%%%%%

 \section{Polynomials of  degree multiple of 4}\label{section:main}

We collect in this section
results concerning all
polynomials of degree $m$ congruent to 0 modulo 4
as opposed to more specific 
degrees $m=2^r(2^{\ell}+1)$ treated in Section \ref{section:specific_m}.

%%%%%%%%%%%%%%%%%%%%%%%
%%%%%%%%%% Subsection : Etude de la Trace
%%%%%%%%%%%%%%%%%%%%%%%

\subsection{A unique polynomial $L_{\alpha}f$ such that $(L_{\alpha}f)(x(x+\alpha))=D_{\alpha}f(x)$}

The following proposition is a particular case
of Proposition 2.1, Proposition 2.3
and Lemma 2.5
in \cite{AHV} whose
proofs rest on linear algebra.
\begin{proposition}\label{proposition:Lalpha}
Let $m$ be an integer such that $m \equiv 0 \pmod 4$
and $f=\sum_{k=0}^m a_{m-k} x^{k} \in \mathbb{F}_q[x]$
a polynomial of degree $m$.
Consider $\alpha \in \mathbb{F}_q^*$.
There exists a unique polynomial
 $L_{\alpha}f :=\sum_{k=0}^{d}b_{d-k}x^{k}$ in $\mathbb{F}_q[x]$
of degree less or equal to $d:=(m-2)/2$
such that
  $$(L_{\alpha}f) (x(x+\alpha))=D_{\alpha}f(x).$$
Moreover
\begin{enumerate}[label=(\roman*)]
\item the application $L_{\alpha}$ is linear,
\item  $L_{\alpha}
f$ has degree $d$ if and only if $a_1 \neq 0$,
\item 
$ \left\{
\begin{array}{rcrrr}
b_0 & = & a_1 \alpha &                                             & \\                   
b_1 & = & a_2 \alpha^2 + a_3 \alpha                                & \textrm{if } m \equiv 0 \pmod 8  & 
 \textrm{or} \\ 
b_1 & = & a_0 \alpha^4 + a_1 \alpha^3  + a_2 \alpha^2 + a_3 \alpha & \textrm{if } m \equiv 4 \pmod 8  &
\end{array}
\right .
$
\item when seen as an element of $\mathbb{F}_2[a_0,\ldots,a_m,\alpha]$
the coefficient $b_i$ is an homogeneous polynomial
of degree $2i+2$ when considering the weight $w$
such that $w(a_j)=j$ and $w(\alpha)=1$.
\end{enumerate}

\end{proposition}
\begin{remark}
In order to fullfil condition (I.c)
we want the polynomial $L_{\alpha}f$ 
to have odd degree.
But $d=(m-2)/2$ is odd
when $m \equiv 0 \pmod 4$,
so a sufficient condition is to take a nonzero $b_0$ 
i.e. $a_1 \neq 0$.
\end{remark}

\begin{remark}
The point of view 
where $L_{\alpha}f$ is thought as an element
of $\mathbb{F}_2 [a_0,\ldots, a_m,\alpha]$
will often be adopted in the following.
The proof of the existence of a unique $L_{\alpha}f \in \mathbb{F}_2 [a_0,\ldots, a_m,\alpha]$
such that $(L_{\alpha}f)(x(x+\alpha))=D_{\alpha}f(x)$
rests on the same arguments:
the equation reduces to a unit triangular system.
\end{remark}

\subsection{The trace condition}

Recall that condition (II)
involves the existence of a solution of the equation $x^2+ \alpha x= b_1/b_0$
in $\mathbb{F}_{2^n}$.
By Hilbert's Theorem 90
this is equivalent to say that 
$\Tr \left(\frac{b_1}{b_0\alpha^2}\right)=0$
where $\Tr$ 
stands for the trace from $\mathbb{F}_{2^n}$ 
to $\mathbb{F}_{2}$.

\begin{proposition}\label{proposition:trace}
Let $f=\sum_{k=0}^{m}a_{m-k}x^{k}\in \F_{2^{n}}[x]$
 be a polynomial of degree $m$ such that $a_1\neq 0$.
 
 \begin{enumerate}[label=(\roman*)]
\item If $m \equiv 0 \pmod 8$ 
then the  number of $\alpha \in \F_{2^n}^{*}$ such that $\Tr\left(\frac{b_1}{b_0 \alpha^2}\right)=0$ is $2^{n-1}-1$ if $a_{2}^{2} +a_1 a_3 \neq 0$ and $2^{n}-1$ otherwise.
\item If $m \equiv 4 \pmod 8$ 
then the number of $\alpha\in \F_{2^n}^*$ such that $\Tr\left(\frac{b_1}{b_0 \alpha^2}\right)=0$ is
 equal to $2^{n}-1$ if $a_2^2+a_1a_3=0$ 
 and greater or equal to
$
\frac{1}{2}
(2^n-2^{n/2+1}-1)$ otherwise.
\end{enumerate}
\end{proposition}

\begin{proof}
%%%%%%%%% Proof m=0 [8]
The situation is simpler when
$m \equiv 0 \pmod 8$.
We have by Proposition \ref{proposition:Lalpha} that  
$ \frac{b_1}{b_0}=\frac{a_2 \alpha + a_3}{a_1}$
so
$\Tr\left(\frac{b_1}{b_0 \alpha^2}\right)=\Tr\left(\frac{a_2^2 +a_1a_3}{a_1^2 \alpha^2}\right)$.
We notice that if $a^2_2  +a_1a_3 \neq 0$
then the map $\alpha \mapsto \frac{a_{2}^{2} +a_1 a_3}{a_1 ^2 \alpha^2}$ is a permutation of $\F_{2^{n}}^{*}$.

%%%%%%%%% Proof m=0 [4]
In the case where $m \equiv 4 \pmod 8$,
we find that
$\Tr \left(\frac{b_1}{b_0\alpha^2}\right)=0$ if and only 
if $\Tr \left(\frac{a_2^2+a_1a_3}{a_1^2\alpha^2}+\frac{a_0\alpha}{a_1}\right)=n$.
If $a_2^2+a_3a_1=0$ then there exist $2^{n-1}-1$ elements $\alpha$ such
that $\Tr\left(\frac{a_0}{a_1}\alpha\right)=n$.
Otherwise, let us set $C=a_0/a_1$ and $D^2=\frac{a_2^2+a_1a_3}{a_1^2}$, 
so we are reduced to study the equation
$\Tr (C\alpha)+\Tr(D/\alpha)=n$.
Then if we set $K^2=CD$ and $v=a_0\alpha/a_1K$, we obtain
$\Tr (Kv)+\Tr (K/v)=n$. 
Choosing $S$ with $\Tr (S)=n$ we can rewrite the last condition as the existence of
$w$ such that $Kv + K/v = S + w^2 +w$ and multiplying through by $v^2$ and setting $y = vw$ turns the equation
into $K(v^3+v) +Sv^2 = y^2 + vy$, which defines an  elliptic curve $E$ (as $K \ne 0$) whose projective closure is smooth with one point
at infinity.
Let us set $q=2^n$.
By the Hasse-Weil bound, 
the number of rational points over ${\mathbb F}_q$ 
on $E$ is at least $q -2\sqrt{q}$.
Moreover, for any $v$ in ${\mathbb F}^{\ast}_{q}$ there are 
at most 2 elements
 $(v,w)$ on $E$. 
Therefore, there are at 
least $\frac{1}{2}(q-2\sqrt{q}-1)$ 
suitable nonzero $v$ and thus as many $\alpha$ which enables us to conclude the proof.
\end{proof}

%%%%%%%%%%%%%%%%%%%%%%%
%%%%%%%%%% Subsection : valeurs critiques non degenerees
%%%%%%%%%%%%%%%%%%%%%%%

\subsection{Nondegenerate critical points}
In this subsection we want to bound
for a given polynomial $f$ 
the number of $\alpha$ such that 
$L_{\alpha}f$ has degenerate
critical points,
that is those which do not satisfy condition (I.a).

\begin{proposition}\label{proposition:ncdp}
Let $m$ be an integer such that $m \equiv 0 \pmod 4 $ and
let $f=\sum_{k=0}^{m}a_{m-k}x^{k}\in \F_{2^{n}}[x]$
 be a polynomial of degree $m$ 
such that $a_1\neq0$. 
The critical points of $L_{\alpha}f$ are nondegenerate except 
for at most $(m-1)(m-4)$ values of $\alpha \in \overline{\F}_{2}$.
\end{proposition}

\begin{proof}
Recall that
the second Hasse-Schmidt derivative $(D_{\alpha}f)^{[2]}$ of the polynomial $D_{\alpha}f$
is defined by 
$$D_{\alpha}f(t+u)\equiv D_{\alpha}f(t) + (D_{\alpha}f)'(t)u + (D_{\alpha}f)^{[2]}(t)u^2 \pmod{u^3}$$
and that
 Lemma 3.3 of \cite{AHV}
states that the critical points of
$L_{\alpha}f$ are nondegenerate 
if and only if 
$(D_{\alpha}f)'$ and $(D_{\alpha}f)^{[2]}$ have no common roots
in $\overline{\mathbb{F}}_2$.

We note $h=f-a_0x^m$.
The congruence of $m$ implies that 
 $(x^m)'$ and $(x^m)^{[2]}=\binom{m}{2} x^{m-2}$ both vanish.
Thus we use
the linearity of $D_{\alpha}$ and of the derivative operators
to get 
$(D_{\alpha}f)' = (D_{\alpha}h)'$ and
$(D_{\alpha}f)^{[2]} = (D_{\alpha}h)^{[2]}$
and so the equality between the two resultants 
$\Res((D_{\alpha}f)', (D_{\alpha}f)^{[2]}) = 
\Res((D_{\alpha}h)', (D_{\alpha}h)^{[2]})$.
As $a_1 \neq 0$,
we are reduced to the case where the degree of $h$ is congruent to $3$ modulo $4$.
This case is treated in Lemma 3.4 in \cite{AHV} which states that
$\Res((D_{\alpha}h)', (D_{\alpha}h)^{[2]})$ is a polynomial of degree
$(m-1)(m-4)$ in $\alpha$ with at most $(m-1)(m-4)$ roots in $\overline{\mathbb{F}}_2$.
\end{proof}

%%%%%%%%%% Subsection : valeurs critiques distinctes
%%%%%%%%%%%%%%%%%%%%%%%

\subsection{Distinct critical values.}\label{subsection:dcv}
In this section we will bound the number of $\alpha\in \F_{2^n}^{\ast}$ 
such that the condition (I.b) is not satisfied.

The aim of the two following statements
is to study the case of $L_{\alpha} (x^{m-1})$ when $m\equiv 0 \pmod 4$.

\begin{lemma}\label{lemma:link_between_roots}
We consider an integer 
$m \geqslant 8$ 
 such that 
  $m \equiv 0 \pmod 4$ and  $d=(m-2)/2$.
 For all $f=\sum_{k=0}^{m}a_{m-k}x^{k} \in \F_{2^n}[x]$ of degree $m$ such that $a_1\neq0$
 the following conditions are equivalent:
\begin{enumerate}[label=(\roman*)]
\item $(L_{\alpha}f)'$ has $\frac{d-1}{2}$  distinct (double) roots  in $ \overline{{ \mathbb{F}}}_2$.  
\item $(D_{\alpha}f)'$ has $d-1$ distinct (double) roots  in $ \overline{{ \mathbb{F}}}_2$.  
\end{enumerate}
\end{lemma}
\begin{proof}
Let us first assume that 
$\tau_{1},\tau_{2},\ldots,\tau_{(d-1)/2}$ 
are $\frac{d-1}{2}$ distinct roots of $(L_{\alpha}f)'$. 
We have that
\begin{equation}\label{equation:derivee_de_Dalpha}
(D_{\alpha}f)' = \left( L_{\alpha} f \circ T_{\alpha}  \right)'= \alpha (L_{\alpha}f)'\circ T_{\alpha}
\end{equation}
where $T_{\alpha}(x):=x(x+\alpha)$.
So if we choose $z_i\in \overline{\F}_{2}$
such that $T_{\alpha}(z_i)=T_{\alpha}(z_i+\alpha)=\tau_{i}$
the elements $z_1,z_1+\alpha,z_2,z_2+\alpha,\ldots,z_{\frac{d-1}{2}},z_{\frac{d-1}{2}}+\alpha$
 are $d-1$ distinct roots of $(D_{\alpha}f)'$.

Conversely, $z$ is a root of $(D_{\alpha}f)'= \alpha (L_{\alpha}f)'\circ T_{\alpha}$ if and only if $z+\alpha$ is.
So $d-1$ distinct roots of $(D_{\alpha}f)'$
 can always be written as $z_1,z_1+\alpha,z_2,z_2+\alpha,\ldots,z_{\frac{d-1}{2}},z_{\frac{d-1}{2}}+\alpha$.
 Now set $\tau_i:=T_{\alpha}(z_i)=T_{\alpha}(z_i+\alpha)$ to get
 $\frac{d-1}{2}$ distinct roots of $(L_{\alpha}f)'$.
\end{proof}

\begin{lemma}\label{lemma:case_of_monomial}
We consider an integer  
$m \geqslant 8$ 
such that 
$m \equiv 0 \pmod 4$
and the monomial $f(x)=x^{m-1}$.
For any $\alpha\in \F_{2^n}^{\ast}$
the polynomial
$(L_{\alpha}f)'$ has $(d-1)/2$ distinct double roots in $\overline{\F}_{2}$,
namely the $\tau_1,\ldots,\tau_{(d-1)/2}$ defined by
$$\tau_{i}=\frac{\alpha^2}{1+\theta_{i}}+\frac{\alpha^2}{1+\theta_{i}^2}$$ 
where $\theta_{1},\ldots,\theta_{(d-1)/2}$ are $(d-1)/2$
different
$d$-th roots of the unity in 
$\overline{\F}_{2}\setminus\{ 1\}$
such that 
$\theta_{i} \theta_{j}\neq1$ for $i\neq j.$
\end{lemma}

\begin{proof}
We fix $\alpha\in \F_{2^n}^{\ast}$. 
By the previous lemma, it is sufficient to determine
the roots of $(D_{\alpha}f)'$, that is
the solutions of the equation
$x^{m-2}+(x+\alpha)^{m-2}=0$.
As these solutions are obviously different from $0$, 
it amounts to studying the solutions 
$\theta$
of $\left( \frac{x+\alpha}{x} \right)^{m/2-1}=1$.
For $\theta=1$ there is no corresponding solution $x$,
but for any other 
$(m/2-1)$-th root
of the unity $\theta$
there is one and only one solution $x= \frac{\alpha}{1+\theta}$.
To conclude that 
$(L_{\alpha}f)'$ has 
$m/4-1$ (that is $(d-1)/2$)
distinct roots of the claimed form, 
we use the equality (\ref{equation:derivee_de_Dalpha}) and the fact that
for $x=\frac{\alpha}{1+\theta}$ and $x'=\frac{\alpha}{1+\theta'}$, we have $T_{\alpha}(x)=T_{\alpha}(x')$ if and only if $\theta\theta'=1$.
\end{proof}

Following the Appendix of Geyer in \cite{JardenRazon}
we associate to any polynomial $g=\sum_{k=0}^{d}b_{d-k} x^{k} \in \mathbb{F}_q[x]$
of degree $d$
 a nonzero rational function $\Pi \in \mathbb{F}_2[b_0,\ldots,b_d][1/b_0]$
whose zeros correspond exactly to the polynomials with non-distinct critical values
(note that in \cite{JardenRazon}
$\Pi$ is actually a polynomial 
in $\mathbb{F}_2[b_0,\ldots,b_d]$
as $g$ is supposed there to be monic).
We will use the notation $\Pi_d$ to stress the dependance on the degree.
Recall that $\Pi_d$ is defined as follows in \cite{JardenRazon}
 \begin{equation}\label{equation_pi}
 \Pi_d(g):= \prod_{i \neq j} \left( g(\tau_i) - g(\tau_j) \right) 
 \end{equation}
 where the $\tau_i$ are the  (double) roots of $g'$.
In the following lemma,
which is an adaptation of Lemma 3.8 in \cite{AHV}
to handle the case
when $m \equiv 0 \pmod 4$,
we prove
that for a well chosen value of $N$ 
the rational function
$b_0^{N}\Pi_d \left( L_{\alpha}f \right) $
becomes a polynomial in $\mathbb{F}_2 [a_0,\ldots,a_m,\alpha]$
with useful 
homogeneity properties.

\begin{lemma}\label{lemma:informations_about_Pi}
Let $m\geqslant 8$ 
be an integer
such that $m\equiv 0 \pmod 4$.
We consider a degree $m$ polynomial
$f=\sum_{k=0}^{m}a_{m-k}x^{k}$ such that $a_1 \neq 0$
and the associated polynomial 
$L_{\alpha} f=\sum_{k=0}^{d}b_{d-k}x^{k}$.
If we set $d=(m-2)/2$ and $e=\binom{(d-1)/2}{2}$
 then
$b_{0}^{de} \Pi_{d} (L_{\alpha}f)$ 
is a polynomial in $\F_{2}[a_0,\ldots,a_m,\alpha]$
 each of whose terms contains a product of $(d+2)e$ terms $a_i$. 
This polynomial is 
homogeneous of degree $(6d+4)e$ 
when considering the weight $w$ such that $w(\alpha)=1$ and $w(a_i)=i.$
\end{lemma}
\begin{proof}
Mutatis mutandis,
the proof can be read off  
from Lemma 3.8 in \cite{AHV}.
\end{proof}

%%%%%%%%%%%%%%%%
%%%%%%%%%%%%%%%

\section{Polynomials of degree $2^r(2^{\ell}+1)$.}\label{section:specific_m}

Our strategy is now 
to prove that $b_0^{de}\Pi_d \left( L_{\alpha}f \right) $
has a simple leading coefficient 
when seen as a polynomial in $\alpha$ and when $f$ has degree $2^r(2^{\ell}+1)$.
First,
the following proposition gives an handy interpretation
which involves
the trace polynomials $P_{k}$ 
defined by
 $$P_{k}(x):=x+x^2+\cdots+x^{2^{k-1}}$$
for any integer $k \geqslant 1$.

\begin{proposition}\label{proposition:nbe_alpha_cvd}
Let 
 $r \geqslant 2$ and  $\ell\geqslant1$.
   We set $m=2^r(2^{\ell}+1)$,
 $d=\frac{m-2}{2}$ and  $e=\binom{(d-1)/2}{2}$.
We consider a  polynomial
 $f(x)=\sum_{k=0}^{m}a_{m-k}x^k$ 
 of degree $m$
 such that $a_1 \neq 0$
 and the associated polynomial
$L_{\alpha} f=\sum_{k=0}^{d}b_{d-k}x^{k}$ 
as above.

Recall that we  use the notation  
$\Pi_d(g)$ for the rational function
defined in  (\ref{equation_pi})
which describes the locus of polynomials $g$ with non-distinct critical values.
Thus
\begin{enumerate}[label=(\roman*)]
\item the indeterminate $a_0$ appears in the 
polynomial $b_{0}^{de} \Pi_{d} (L_{\alpha}f)$
with a power at most $2e$.
\item When seen as an element of  $\mathbb{F}_{2} [a_0,\ldots,a_m][\alpha]$
the polynomial $b_{0}^{de} \Pi_{d} (L_{\alpha}f)$
has degree at most
$(5d+4)e$.
Moreover, the only monomial of such a degree that can appear in 
$b_{0}^{de} \Pi_{d} (L_{\alpha}f)$ is
$a_0^{2e}
a_1^{de}
\alpha^{(5d+4)e}$.
 \item 
 When seen as an element of  $\mathbb{F}_{2} [a_0,\ldots,a_m][\alpha]$
the polynomial $b_{0}^{de} \Pi_{d} (L_{\alpha}f)$
has degree exactly
 $(5d+4)e$
if and only if 
for any choice of different roots $\tau_i$ and $\tau_j$ of $L_1{(x^{m-1})}'$
we have $P_{\ell}(\tau_i + \tau_j) \neq 0$
where $P_{\ell}$ is the $\ell$-th trace polynomial.
\end{enumerate}
\end{proposition}

\begin{proof}

We are first looking for the coefficient $a_0$
in the polynomial 
\begin{equation}\label{polynomePi}
b_0^{de}\Pi_d (L_{\alpha}f) = 
(a_1 \alpha)^{de} \prod_{i<j}
\left( 
\sum_{k=0}^{d}b_{d-k}^2(\tau_{i}^{2k}+\tau_{j}^{2k})
\right).
\end{equation}
Our point of departure is
that 
when $m$ admits the special form $m=2^r(2^{\ell}+1)$ then
$L_{\alpha} (x^{m})$ has the following  very simple expression
\begin{equation}\label{equation:L1xm}
L_{\alpha}(x^m)=\alpha^m+\sum_{k=0}^{\ell-1} \alpha^{m-2^{r+k+1}} x^{2^{r+k}}.
\end{equation}
It can be proved easily by checking
that the composition of the right hand side 
with $x(x+\alpha)$
is actually $x^m+(x+\alpha)^m$.

Then the linearity of $L_{\alpha}$ yields
 that when $f=a_0 x^{m}+\cdots$ 
the indeterminate $a_0$ appears in few coefficients $b_i$ 
of $L_{\alpha} f$,
namely in $b_d, b_{d-2^{r}}, b_{d-2^{r+1}},\ldots, b_{d-2^{r+\ell-1}}.$

Actually $b_d$ does not contribute
in the product (\ref{polynomePi})
since the terms 
$\tau_{i}^{2k}+\tau_{j}^{2k}$
simplify for $k=0$
in the sum between parentheses.
Also, the terms  $\tau_i$
do not give rise to the indeterminate $a_0$. 
Indeed, any monomial in the $b_i$ in (\ref{polynomePi})
will be 
multiplied 
 by a polynomial $Q$
in the variables $\tau^2_1,\ldots,\tau^2_{(d-1)/2}$
and this polynomial is
invariant under the action of the symmetric group ${\frak{S}}_{(d-1)/2}$.
But
 $$\left( L_{\alpha} f \right)' (x)= b_0x^{d-1}+b_2x^{d-3}+\cdots+b_{d-3}x^2+b_{d-1} 
 = b_0(x^2+\tau_1^2)\cdots(x^2+\tau_{(d-1)/2}^2)$$
thus $Q$ 
 belongs to $\mathbb{F}_2 [\frac{b_2}{b_0},\frac{b_4}{b_0},\ldots,\frac{b_{d-1}}{b_0}]$.
Again, the indeterminate $a_0$
does not appear in $Q$ %by Lemma \ref{lemma:Lalphadexm}
%and because 
as
${d-2^{r}}, {d-2^{r+1}},\ldots, {d-2^{r+\ell-1}}$
are odd.

So to investigate where the largest power of $a_0$ appears in
(\ref{polynomePi})
 we are reduced to study
\begin{equation}\label{terme_a_etudier} 
 (a_1 \alpha)^{de} \prod_{i<j}
\left(b_{d-2^{r}}^{2}(\tau_{i}+\tau_{j})^{2^r}+b_{d-2^{r+1}}^{2}(\tau_{i}+\tau_{j})^{2^{r+1}}+\cdots
+b_{d-2^{r+l}}^{2}(\tau_{i}+\tau_{j})^{2^{r+\ell-1}}\right).
 \end{equation}
This largest power is bounded by $2\binom{(d-1)/2}{2}=2e$
and point (\textit{i)} is proven.

%%%%%%%%%%%%%%%%%%
To prove point \textit{(ii)},
consider a monomial $a_0^{u_0}a_1^{u_1} \ldots a_m^{u_m} \alpha^v$
arising in 
$b_0^{de}\Pi_d (L_{\alpha}f)$.
Recall that by
Lemma \ref{lemma:informations_about_Pi} we count
$2e+de$ indeterminates $a_i$
 and we have 
$u_1+2u_2+\cdots+mu_m+v=(6d+4)e$.
As we have just proved that $u_0 \leqslant 2e$,
it implies that
either $u_0=2e$, $u_1=de$,
$u_2=u_3=\cdots=0$ and $v=(5d+4)e$, or
in any other case $v<(5d+4)e$.

To treat point (\textit{iii) },
 let us now determine when the monomial 
$a_0^{2e}
a_1^{de}
\alpha^{(5d+4)e}$
does appear in $b_0^{de}\Pi_d (L_{\alpha}f)$.
We use the expression of the coefficients $b_{d-2^{r+k}}$ obtained above to 
transform (\ref{terme_a_etudier}) into
$$ a_{0}^{2e}
\left[ a_{1}^{de}\alpha^{de}
\prod_{i<j}\left(\sum_{k=0}^{\ell-1}
(\alpha^{m-2^{r+k+1}})^{2}(\tau_{i}+\tau_{j})^{2^{r+k}}\right) \right]. $$
We know that the expression between brackets is a polynomial in
$\mathbb{F}_2[a_0,\ldots,a_m][\alpha]$ 
with no term $a_0$ and we wonder
if the monomial $a_1^{de} \alpha^{(5d+4)e}$ does appear.
To this end it is sufficient to evaluate it
when $\alpha=1$, $a_1=1$ and $a_2=\cdots=a_m=0$
that is to consider
$ \displaystyle{
\prod_{i<j} 
 P^{2^r}_{\ell} 
 \left(
 \tau_{i}+\tau_{j}
 \right)
  } $
 where the $\tau_i$ are the $(d-1)/2$ different roots of
 $L_{1}(x^{m-1})'$ 
 which are described in Lemma (\ref{lemma:case_of_monomial}).
 It concludes the proof.
 \end{proof}
 
The Proposition \ref{proposition:Pl_taui_plus_tauj_non_nul}
will exploit this interpretation. To make its proof more readable
we  now provide several lemmas, 
the first of them being an easy
arithmetic result.

\begin{lemma}\label{lemma:pgcd}
Let  $r \geqslant 2$ and  $\ell\geqslant1$.
We set $m=2^r(2^{\ell}+1)$ and $d=(m-2)/2$.
\begin{enumerate}[label=(\roman*)]
\item  If $\gcd(r,\ell)=1$ then $\gcd(d,2^{2 \ell}-1)=1$,
\item If $\gcd(r,\ell)=2$ then $\gcd(d,2^{2 \ell}-1)=3$.
\end{enumerate}
 \end{lemma}
 \begin{proof} 
For example,
start with the observation that
if $t$ divides $2^{2 \ell}-1$ then one can write $t=ab$ where 
$a$ and $b$ are divisors of  $2^{\ell}-1$ and $2^{\ell}+1$.
Thus work modulo $a$ and $b$
and  study the order of $2$ in $(\mathbb{Z}/b\mathbb{Z})^{\ast}$.
\end{proof}

We have compiled in the following lemma
some basic computational results
which will prove useful establishing
Proposition \ref{proposition:Pl_taui_plus_tauj_non_nul}.
\begin{lemma}\label{lemma:calculs_Pl}
Fix two integers   $r \geqslant 2$ 
and  $\ell \geqslant1$ and  set $m=2^{r}(2^{\ell}+1)$.
We have
\begin{enumerate}[label=(\roman*)]
\item  $ L_{1}(x^{m-1}) = 
x^{2^r-1} + \left( 1+ \sum_{k=r}^{r+\ell-1} x^{2^{k}}
 \right) 
\sum_{k=0}^{r-1} x^{2^{k}-1}$,
\item $x^2(L_{1}(x^{m-1}))' = P_{r}^2(x)+P_{\ell}^{2^r}(x)P_{r-1}^{2}(x)$.
\item If $\tau_i$ is a root of $L_1(x^{m-1})'$ then    $P_{r-1}(\tau_{i}) \neq 0$.
\item For any choice of different roots $\tau_i$ and $\tau_j$ of $L_1(x^{m-1})'$ 
such that  $P_{\ell}(\tau_i+\tau_j) = 0$ we have 
$P_{r-1}(\tau_{i}+\tau_j) \neq 0$ and 
$$P_{\ell}(\tau_{i})^{2^{r-1}} =\frac{P_{r}(\tau_{i})}{P_{r-1}(\tau_{i})}=\frac{P_{r}(\tau_{i}+\tau_{j})}{P_{r-1}(\tau_{i}+\tau_{j})}=
\frac{P_{r}(\tau_{j})}{P_{r-1}(\tau_{j})}=P_{\ell}(\tau_{j})^{2^{r-1}}.$$
\end{enumerate}
 \end{lemma}
 \begin{proof}
By definition of $L_1$  it is sufficient 
for the first point
to compute the composition of the right hand side of the equality and to find $D_1(x^{m-1})$, 
that is $(x+1)^{m-1}+x^{m-1}$.

Just differentiate the former equality and use
the relation  
 $P_{r-1}^2 (x) + x^{2^r}=P_{r}^2(x)$
 to get the second point.

If $P_{r-1}$ vanishes at a root $\tau_i$ of $L_1(x^{m-1})'$, point \textit{(ii)}
leads to $P_{r}(\tau_i)=0$. But the relation between $P_{r-1}$ and $P_r$ above
implies that $\tau_i=0$, a contradiction with Lemma \ref{lemma:case_of_monomial}.

Under the hypotheses of the last item, 
point \textit{(ii)} applies to get 
$P_{r}^2(\tau_i)+P_{\ell}^{2^r}(\tau_i)P_{r-1}^{2}(\tau_i)=0$
and thus $P_{\ell}(\tau_{i})^{2^{r-1}} =\frac{P_{r}(\tau_{i})}{P_{r-1}(\tau_{i})}$,
 one of the claimed results. 
Now adding the equalities \textit{(ii)} for $x=\tau_i, \tau_j$ 
and using $P_{\ell}(\tau_i)=P_{\ell}(\tau_j)$ lead to
$P_r^2(\tau_i+\tau_j)+P_l^{2{r}}(\tau_i)P_{r-1}^2(\tau_i+\tau_j)=0$.
Again, if $P_{r-1}(\tau_i+\tau_j)=0$ we obtain $P_r(\tau_i+\tau_j)=0$
and thus $\tau_i=\tau_j$, which is impossible.
We have just proven that $P_{r-1}(\tau_i+\tau_j)$ is nonzero
and so we can write 
$P_{\ell}(\tau_{i})^{2^{r-1}} =
\frac{P_{r}(\tau_{i}+\tau_{j})}{P_{r-1}(\tau_{i}+\tau_{j})}$.
 \end{proof}
 
\begin{proposition}\label{proposition:Pl_taui_plus_tauj_non_nul}
We fix  $r \geqslant 2$ and  $\ell\geqslant1$ and %
we set $m=2^r(2^{\ell}+1)$.
Recall that we denote by  $P_{k}$ the $k$-th trace polynomial.
Thus
 $\gcd(r,\ell) \leqslant 2$ if and only if
for any choice of different roots $\tau_i$ and $\tau_j$ of $L_1{(x^{m-1})}'$
we have $P_{\ell}(\tau_i + \tau_j) \neq 0$.

\end{proposition} 
\begin{proof}
We set  $m=2^r(2^{\ell}+1)$ with
$r \geqslant 2$ and  $\ell\geqslant1$ such that $\gcd({r,\ell}) \leqslant 2$.
Let  $\tau_i$ and $\tau_j$ be two different roots of $L_1(x^{m-1})'$
such that $P_{\ell}(\tau_i + \tau_j)=0$.
A classical property of the trace polynomials 
imply that $\tau_i + \tau_j$ belongs to $\mathbb{F}_{2^{\ell}}$ and so does
$\frac{P_{r}(\tau_i+\tau_j)}{P_{r-1}(\tau_i+\tau_j)}$.
(Remember that by  Lemma \ref{lemma:calculs_Pl}
$P_{r-1}(\tau_i+\tau_j)$ is nonzero.) 

By point \textit{(iii)} of Lemma \ref{lemma:calculs_Pl} again, 
 we know that $P_{\ell}(\tau_{i})^{2^r}$ lies in $\mathbb{F}_{2^{\ell}}$
 and thus $P_{\ell}(\tau_{i})$, too.
 Now we use the expression 
$\tau_{i}=\frac{1}{1+\theta_{i}}+\frac{1}{1+\theta_{i}^2}$
given by Lemma \ref{lemma:case_of_monomial}
where $\theta_i$  is a
$d$-th root of the unity in 
$\overline{\F}_{2}\setminus\{ 1\}$.
Substituting into $P_{\ell}(\tau_{i}+\tau_{i}^{2^{\ell}})=0$ 
leads to a telescopic sum which simplifies into
$\frac{1}{1+\theta_{i}}=\frac{1}{1+ \theta_{i}^{2^{2\ell}}}$
which gives $\theta_{i}^{2^{2\ell}-1}=1$.
Since 
$\theta_i^d=1$
 it follows that $\theta_i^{\gcd(2^{2\ell}-1,d)}=1$.

If $\gcd(\ell,r)=1$ then Lemma \ref{lemma:pgcd} gives
$\gcd(2^{2\ell}-1,d)=1$, so $\theta_i=1$, a contradiction.

We now turn to the case $\gcd(\ell,r)=2$.
This time Lemma \ref{lemma:pgcd} asserts $\gcd(2^{2\ell}-1,d)=3$.
We deduce that  $\theta_i \in \mathbb{F}_4$ 
and so $\tau_i \in \mathbb{F}_4$.
As $\ell$ is even, we deduce that $\tau_i \in \mathbb{F}_{2^{\ell}}$ and
in consequence, $P_{\ell} (\tau_i) \in \mathbb{F}_2$.
One can show that $P_{\ell}(\tau_i) \neq 1$,
otherwise point \textit{(ii)} of
Lemma \ref{lemma:calculs_Pl}
would imply $P_{r-1}(\tau_i)=P_r(\tau_i)$
and thus $\tau_i=0$,
a contradiction
with the expression of $\tau_i$ given by Lemma
\ref{lemma:case_of_monomial}.
So $P_{\ell}(\tau_i)=0$.

Now point \textit{(iii)} of Lemma  \ref{lemma:calculs_Pl} also implies that 
$P_r(\tau_i)=0$.
But for any $u \in \mathbb{F}_4 \setminus \{0,1\}$ 
and for any positive even integer $k$
we have $P_k(u)=0$ if and only if $4$ divides $k$,
as $P_k(u)= \sum_{s=0}^{k/2-1} (u+u^2)^{2^{2s}}$.
We know that $\tau_i \neq 0$
and when starting the proof with two different roots
$\tau_i$ and $\tau_j$  of $L_1(x^{m-1})'$, 
we could have choosen $\tau_i \neq 1$.
Consequently $4$ divides $\ell$ and $r$, a contradiction.

Conversely, suppose that $a:=\gcd(r,\ell) \geqslant 3$.
Since $a$ divides $r$ and $\ell$
then  $P_a$ divides $P_{r}$ and $P_{\ell}$
(juste write $P_{ab}=\sum_{k=0}^{b-1} P_a^{2^{ka}}$).
As $a \geqslant 3$ and $P_a$ is separable
then  there exist $\tau_i$ and $\tau_j$ 
two different nonzero roots of $P_a$
in $\overline{\mathbb{F}}_2$.
From point \textit{(ii)} of Lemma \ref{lemma:calculs_Pl}
we deduce that $\tau_i$ and $\tau_j$ are roots of $L_1(x^{m-1})'$.
But by linearity we also have $P_{\ell}(\tau_i+\tau_j)=0$ and we are done.
\end{proof}

We can now bring together
these different ingredients 
to bound the number of $\alpha$ such that 
 $L_{\alpha}f$ does not have distinct critical values, 
that is such that
condition I.b
fails.

\begin{theorem} 
Let 
 $r \geqslant 2$ and  $\ell\geqslant1$.
We set $m=2^r(2^{\ell}+1)$, $d=(m-2)/2$ and $e=\binom{(d-1)/2}{2}$.
If $\gcd(r,\ell)=1$ or $2$ then
for any positive integer $n$, for any choice of the coefficients $a_i$ in  $\mathbb{F}_{2^n}$ 
such that $a_0 \neq 0$ and $a_1 \neq 0$, 
the number of $\alpha\in \overline{\F}_{2}^{\ast}$
such that the polynomial $L_{\alpha}f$ associated to 
 $f(x)=\sum_{k=0}^{m}a_{m-k}x^k$ 
 has not distinct critical values is 
  at most $(5d+4)e.$ 
\end{theorem}

\begin{proof}
Fix $r, \ell$ and then $m$ as in the statement of the theorem.
We know by Lemma \ref{lemma:case_of_monomial}
that $L_1(x^{m-1})'$ has $(d-1)/2$ distinct roots 
$\tau_1, \tau_2,\ldots$ in $\overline{\mathbb{F}}_2$
and that by Proposition \ref{proposition:Pl_taui_plus_tauj_non_nul}
if $i \neq j$ then $P_{\ell}(\tau_i + \tau_j) \neq 0$.
Now for $f(x)=\sum_{k=0}^m a_{m-k} x^k \in \mathbb{F}_2 [a_0,\ldots,a_m][x]$
Proposition \ref{proposition:nbe_alpha_cvd} ensures that
$b_0^{de}\Pi_d (L_{\alpha}f)$,
seen  as an element of $\mathbb{F}_2 [a_0,\ldots,a_m][\alpha]$,
has degree exactly $(5d+4)e$ and that its leading term is
$a_0^{2d}a_1^{de} \alpha^{(5d+4)e}$.

Last, if we consider  $n \geqslant 1$
and coefficients $a_i$ in $\mathbb{F}_{2^n}$ 
with $a_0, a_1 \neq 0$,
it follows that 
$b_0^{de}\Pi_d (L_{\alpha}f)$
is a polynomial in $\mathbb{F}_{2^n}[\alpha]$ of degree $(5d+4)e$
whose number of roots is bounded by its degree. 
\end{proof}
 
%%%%%%%%%%%%%%%%%%%%%%%%%%%%%%
%%%%%%%%%%%%%%%%%%%%%%%%%%%%%%
%%%%%%%%%%%%%%%%%%%%%%%%%%%%%%
%%%%%%%%%%%%%%%%%%%%%%%%%%%%%%
%%%%%%%%%%%%%%%%%%%%%%%%%%%%%%
%%%%%%%%%%%%%%%%%%%%%%%%%%%%%%
%%%%%%%%%%%%%%%%%%%%%%%%%%%%%%
%%%%%%%%%%%%%%%%%%%%%%%%%%%%%%
%%%%%%%%%%%%%%%%%%%%%%%%%%%%%%
%%%%%%%%%%%%%%%%%%%%%%%%%%%%%%

\section{Proof of the main theorem} \label{subsection:preuve_generale}
We are finally in a position to prove the main result of this paper.

\begin{theorem}\label{theorem:main}
Let  $m=2^{r}(2^{\ell}+1)$ where $\gcd(r,\ell)\leqslant 2$  and  $r \geqslant 2$, $\ell\geqslant 1$.
For $n$ sufficiently large,
for all polynomials $f=\sum_{k=0}^{m}a_{m-k} x^{k}\in{\mathbb F}_{2^n}[x]$
of degree $m$ such that $a_1\not=0$ 
 the differential uniformity $\delta(f)$ is maximal
that is $\delta(f)=m-2$.

In particular, no polynomial
$f=\sum_{k=0}^{m}a_{m-k} x^{k} \in \overline{\mathbb{F}}_2[x]$
of such a degree and
 such that $a_1 \neq 0$
can be  exceptional APN.
\
\end{theorem}
\begin{proof}
Under the hypotheses of the theorem we
set $m=2^{r}(2^{\ell}+1)$
and $d= (m-2)/2$.
Recall that, in order to apply the  Chebotarev density theorem,
our goal is to obtain a regular extension $\Omega/\mathbb{F}_{2^n}(t)$
and that our strategy 
is to bound the number of $\alpha$ such that the different conditions 
sum up in Figure (\ref{diagram}) fail.

Consider an integer $N_1$
such that $n \geqslant N_1$ implies
\begin{equation}\label{bound_Morse}
 \underbrace{\frac{1}{2} \left( 2^n - 2^{n/2+1}-1\right)}_{
\begin{split}
\scriptsize{\textrm{Lower bound for the } } \\
\scriptsize{\textrm{ number of $\alpha$ s. t.}}\\
\scriptsize{\textrm{condition (II) is satisfied}}
\end{split} 
 } 
 > 
\underbrace{(m-1)(m-4)}_{
\begin{split}
\scriptsize{\textrm{Upper bound for the } } \\
\scriptsize{\textrm{ number of $\alpha$ s. t.}}\\
\scriptsize{\textrm{condition (I.a) fails}}
\end{split} 
} 
+ 
\underbrace{(5d+4)
 \binom{(d-1)/2}{2}}_{
\begin{split}
\scriptsize{\textrm{Upper bound for the }  } \\
\scriptsize{\textrm{ number of $\alpha$ s. t.}}\\
\scriptsize{\textrm{condition (I.b) fails}}
\end{split} 
}
.
\end{equation}
Now fix an integer $n \geqslant N_1$ and 
$f=\sum_{k=0}^{m}a_{m-k}x^{k} \in \F_{2^n}[x]$ of degree $m$
such that $a_1 \neq 0$. 
By Proposition \ref{proposition:Lalpha} the associated polynomial
$L_{\alpha}f$ has odd degree $d$, so condition 
(I.c) is satisfied.
Recall that we have proved in Proposition \ref{proposition:trace}
that there are at least
$\frac{1}{2} \left( 2^n - 2^{n/2+1}-1\right)$
values of $\alpha \in \mathbb{F}_{2^n}$ such that
condition (II) is satisfied.
Moreover, 
by Propositions \ref{proposition:ncdp}
and  \ref{proposition:nbe_alpha_cvd}
we know that
$(m-1)(m-4)$ and 
$(5d+4) \binom{(d-1)/2}{2}$ 
respectively bound
the number of $\alpha$ such that
condition (I.a) and (I.b) fail.
So 
one can choose $ \alpha \in \mathbb{F}_{2^n}^{\ast}$
such that conditions (I.a), (I.b) and (II) are satisfied.

For such a choice of $\alpha$ the polynomial
 $L_{\alpha}f$ has nondegenerate critical points,
 distinct critical values,
  an odd degree,
and so is Morse.
Thus 
Proposition 4.2
in the Appendix of Geyer in \cite{JardenRazon}
(which is a form in even characteristic of the Hilbert theorem)
applies and gives
$G=\overline{G}={\frak{S}}_d$ and consequently
the extension $F / \mathbb{F}_{2^n}(t)$ is 
geometric that is with no constant field extension.

Furthermore,
since there exists $x \in \mathbb{F}_{2^n}$ such that $x^2+\alpha x = b_1/b_0$,
 Proposition 4.6 in \cite{AHV}
yields
 $\Gamma=\overline{\Gamma}=(\mathbb{Z} / 2 \mathbb{Z})^{d-1}$ and
the second extension $\Omega / F$ is also
geometric.

Moreover the extension  $\Omega / {\mathbb F}_{2^n}(t)$
is separable since $L_{\alpha}f$ has odd degree and 
is obviously normal as a decomposition field,
so we finally deduce that the 
extension $\Omega / \mathbb{F}_{2^n}(t)$ is a regular Galois extension.

We are now in a position to
 use
the explicit Chebotarev density theorem 
stated as Theorem \ref{Chebotarev}
applied to the regular 
Galois extension  $\Omega / \mathbb{F}_{2^n}(t)$.
Let $V$
 denotes the number of
places  of degree 1
of ${\mathbb F}_{2^n}(t)$
which totally split in 
$\Omega$.
Then denoting by
$d_{\Omega}$ the degree of $\Omega$ over $\mathbb{F}_{2^n}(t)$
and $g_{\Omega}$ the genus of $\Omega$ we get the following lower bound
\begin{equation}\label{bound_Chebotarev}
V \geqslant \frac{{2^n}}{d_{\Omega}}-
\frac{2}{d_{\Omega}}\bigl(g_{\Omega}2^{n/2} + g_{\Omega} + d_{\Omega}\bigr).
\end{equation}
Since $G=\frak{S}_d$ and 
$\Gamma=(\mathbb{Z}/2 \mathbb{Z})^{d-1} $
we have $d_{\Omega} = d! 2^{d-1}$.
Furthermore
$g_{\Omega}\leqslant \frac{1}{2}(\deg D_{\alpha}f-3)d_{\Omega}  +1$
by Lemma 14 of \cite{Pollack} so
$g_{\Omega}\leqslant  d! 2^{d-1} (d-3/2)   +1$.
We deduce the existence of an integer $N_2$  
beyond which 
$V \geqslant 1$.
Thus if $n$ also satisfies $n \geqslant N_2$
 there exists $\beta \in \mathbb{F}_{2^n}$ such that
$D_{\alpha}f(x)=\beta$ has $m-2$ distinct (simple) roots
and $\delta(f)$ is maximal.

As a  straightforward consequence 
we have shown that polynomials of degree $m$ 
with coefficients in a
finite extension of $\mathbb{F}_2$
with a nonzero second leading coefficient cannot be exceptional APN.
\end{proof}

\begin{remark}
Let us make explicit the expression {\sl for $n$ sufficiently large} employed in Theorem \ref{theorem:main}.
As a matter of example, we consider the case of a polynomial $f$ of degree $m=12$  with a nonzero second leading coefficient.
In this case $L_{\alpha}f$ has degree $d=5$, Condition (\ref{bound_Morse}) leads to take $N_1=9$ whereas Condition 
(\ref{bound_Chebotarev}) yields $N_2=28$.
Thus, for $n\geqslant 28$ we have $\delta_{{\mathbb F}_{2^n}}(f)=10$.
As a corollary no polynomial of degree 12 with a nonzero second leading coefficient defined over a finite field of characteristic 2 can be
 exceptional APN. 
\end{remark}

\bigskip
\noindent
{\bf Acknowledgments.}
The authors would like to thank Jos\'e Felipe Voloch for  the proof of Proposition \ref{proposition:trace}.

%%%%%%%%%%%%%%%%%%%%%%%%%
%%%%%%%%%%%%%%%%%%%%%%%%%
%%%%%%%%%%%%%%%%%%%%%%%%%
%%%%%%%%%%%%%%%%%%%%%%%%%%
%%%%%%%%%%%%%%%%%%%%%%%%%%%

\bigskip
\bibliographystyle{plain}
%\bibliography{biblio_ter}

\end{document}